\documentclass{amsart}
\usepackage{etex}
\usepackage[french, english]{babel}
\usepackage{tikz}
\usepackage{tikz-cd}
\usetikzlibrary{matrix,arrows}
\usepackage[all,cmtip]{xy}
\usepackage{enumerate}
\usepackage{dsfont, baskervald}
\usepackage{ marvosym }
\usepackage{mathrsfs}
\usepackage{todonotes}
\usepackage{setspace}
\usepackage{longtable}
\usepackage{tabularx, tabu}
\usepackage{tikz}
\usepackage{amsfonts,amsthm,amsmath,amssymb,comment,stmaryrd,mathalfa,url,mathrsfs, yfonts}
\usepackage{graphicx}
\graphicspath{ {images/} }
\usepackage[section]{placeins}
\usepackage{microtype}
\usepackage{booktabs, calc, multirow}
\usepackage{array}
\usepackage[ocgcolorlinks,backref=page]{hyperref}
\hypersetup{citecolor=blue}
\usepackage{amscd}
\usepackage{makecell}
\usepackage{chngcntr}
\graphicspath{ {images/} }

\usetikzlibrary{%
  matrix,%
  calc,%
  arrows%
}

\def\AA{\mathbb{A}}
\def\CC{\mathbb{C}}

\def\QQ{\mathbb{Q}}
\def\TT{\mathbb{T}}
\def\RR{\mathbb{R}}
\def\ZZ{\mathbb{Z}}
\def \lra{\longrightarrow}

\def \a{\alpha}
\def \e {\epsilon}
\def \aa {\pmb{a}}
\def \A {\mathcal{A}}
\def \g {\gamma}
\def \w {\omega}

\def \W {\mathcal{W}}

\def \l {\lambda}

\def \V {\mathfrak{U}}

\def \Si {\Sigma}
\def \v {v}
\def \ii {\iota}

\def \vv {\text{{\fontfamily{lmr}v\selectfont}}}
\def \d {\mathfrak{d}}
\def \OO {\mathcal{O}}
\def \OOO {\mathscr{O}}

\def \E {\mathcal{E}}

\def \K {U}

\def \M {\mathcal{M}}

\def \CX {\mathscr{X}}
\def \CZ {\mathscr{Z}}
\def \CM {\mathscr{M}}

\def \un{}

\def \ov{\overline}
\def \new {\text{new}}
\def \old {\text{old}}

\def \wh{\widehat}

\def \HHH {\mathcal{H}}
\def \CM {\mathscr{M}}

\def\gothb{\mathfrak{b}}
\def\gothc{\mathfrak{c}}

\def \nn {\mathfrak{n}}
\def \MM {\mathfrak{M}}
\def \k {\kappa}

\def\gothp{\mathfrak{p}}
\def\ps{\mathfrak{p}}
\def\gothq{\mathfrak{q}}
\DeclareMathOperator{\Sch}{Sch}
\DeclareMathOperator{\Set}{Set}
\DeclareMathOperator{\hH}{H}
\DeclareMathOperator{\End}{End}

\DeclareMathOperator{\Sp}{Sp}

\DeclareMathOperator{\val}{val}

\DeclareMathOperator{\Spf}{Spf}

\DeclareMathOperator{\Fred}{Fred}

\DeclareMathOperator{\disc}{Disc}

\DeclareMathOperator{\SL}{SL}
\DeclareMathOperator{\GL}{GL}

\DeclareMathOperator{\Hom}{Hom}

\DeclareMathOperator{\Res}{Res}

\theoremstyle{plain}
\newtheorem{thm}{Theorem}[subsection]
\newtheorem*{thm*}{Theorem}
\newtheorem{lem}[thm]{Lemma}
\newtheorem{cor}[thm]{Corollary}

\newtheorem*{con*}{Conjecture}
\newtheorem{prop}[thm]{Proposition}
\newtheorem*{prop*}{Proposition}
\theoremstyle{definition}
\newtheorem{defn}[thm]{Definition}

\theoremstyle{definition}
\newtheorem{rmrk}[thm]{Remark}
\newtheorem*{rmrk*}{Remark}

\newtheorem*{exmp*}{Example}

\newtheorem*{obs*}{Observation}
\newtheorem{nota}[thm]{Notation}

\title[Overconvergent Jacquet-Langlands correspondence]{The Jacquet-Langlands correspondence for overconvergent Hilbert modular forms}

\author{Christopher Birkbeck}

\begin{document}
	
	\maketitle

	\begin{abstract} We  use results by Chenevier to interpolate the classical Jacquet-Langlands correspondence for Hilbert modular forms, which gives us an extension of Chenevier's results to totally real fields. From this we  obtain an isomorphisms between eigenvarieties attached Hilbert modular forms and those attached to modular forms on a totally definite quaternion algebra over a totally real field of even degree.

	\end{abstract}

\tableofcontents

\section*{\textbf{ Introduction}} 
\counterwithout{thm}{subsection}
 Eigenvarieties are rigid analytic spaces which parametrize systems of hecke eigenvalues of finite slope eigenforms, for this reason they are the subject of much work in number theory. They were first introduced by Colem--Mazur in \cite{cmeig} and Buzzard in \cite{buzeig} and since then many constructions have appeared due to many authors. In particular, we can now construct eigenvarieties associated to any reductive group and one can ask to what extent can we use Langlands functoriality to find maps between different eigenvarieties. Take for example the Jacquet-Langlands correspondence, which tells us (roughly) that the classical spaces of quaternionic modular forms  $S_k^D(N)$ of weight $k$ and level $N$ on a quaternion algebra $D$ of discriminant $\d$ are isomorphic as Hecke modules to the spaces $S_k^{\d \text{-new}}(N \d)$ of $\d$-new modular forms for $\GL_2$. One can then ask if this extends to families of modular forms, i.e., if we can use this to relate the eigenvariety $\CX_D$ coming from  these quaternionic modular forms, to the eigenvariety $\CX_{\GL_2}$ coming from the usual spaces of modular forms. Over $\QQ$, this was answered by Chenevier in \cite{chenjlc}, who showed that there is a closed immersion $\CX_D \hookrightarrow \CX_{\GL_2}$ which interpolates the classical Jacquet-Langlands correspondence\footnote{In general Chenevier proves that one gets a isomorphism onto the $\d$-new `part' of the eigenvariety.}. This result is an instance of what is now called $p$-adic Langlands functoriality.  Other examples of this can be found in work of Hansen \cite{Hansen},  Ludwig \cite{lud},  and Newton \cite{jnew}.

Going back to the result of Chenevier, if one picks $D/\QQ$ to be a totally definite quaternion algebra, then the spaces of overconvergent quaternionic modular forms are very easy to define and work with since there is much simpler geometry involved. This means that if one is interested in studying or computing the action of the $U_p$ operator on the space of overconvergent modular forms, then it is possible to reduce to computing this on spaces of overconvergent quaternionic modular forms. Our goal here is to extend the results by Chenevier to a totally real field $F$. In particular, we have the following theorem.

\begin{thm} 
Let $D/F$ be a totally definite quaternion algebra  of  discriminant $\mathfrak{d}$ defined over a totally real field $F$. Let $p$ be a rational (unramified) prime and $\nn$ an integral ideal of $F$  such that  $p \nmid \nn\d$ and $(\nn,\d)=1$. Let $\CX_D^{red}(\nn p)$  be the nilreduction (see \ref{Red}) of the eigenvariety of level $\nn p$ attached to quaternionic modular forms on $D$. Similarly, let $\CX_{\GL_2}^{red}(\nn \d p)$ denote nilreduction of the eigenvariety associated to cuspidal Hilbert modular forms of level $\nn \d p$ (with the associated moduli problem for this level being representable)  as constructed in \cite{AIP}. Then there is a closed immersion $\iota_D:\CX_{D}^{red}(\nn p) \hookrightarrow \CX_{\GL_2}^{red}(\nn \d p)$ which interpolates the classical Jacquet-Langlands correspondence. Moreover, when $[F:\QQ]$ is even, one can choose $D$ with $\d=1$ so that the above is an isomorphism between the corresponding eigenvarieties.
\end{thm}

Alongside this we also prove following:

\begin{thm}
The eigenvarieties $\CX_D(\nn p)$ and $\CX_{\GL_2}(\nn p) $ as above are reduced, meaing that $\CX_{*}^{red} (\nn p) \cong \CX_{*}(\nn p)$ for $* \in \{D,\GL_2\}$.
\end{thm}

To prove the above results, we first recall the construction of the relevant eigenvarieties, making small modifications so that we can apply Chenevier's Interpolation theorem. In the Hilbert case, we construct the spaces of overconvergent cuspidal Hilbert modular forms following \cite{AIP} and then use Buzzard's eigenmachine to construct the relevant eigenvariety for which we can apply the Interpolation theorem. In the quaternionic case we construct the spaces of overconvergent quaternionic modular forms following \cite{buzeig}, but modifying the weight space so that it is equidimensional of dimension $[F:\QQ]+1$ and therefore matches with the weight space on the Hilbert eigenvariety.

The importance of the above theorems, is that it allows us to study the geometry of the Hilbert eigenvariety by studying the quaternionic eigenvariety which has a more combinatorial nature. Similar ideas have been used in \cite{jacobs,slbd,eovb} to understand the geometry of eigenvarieties attached to modular forms over $\QQ$. Furthermore, the above theorem is one of the key ingredients in \cite{ME}, where it is used to compute slopes of Hilbert modular forms, which in turn gives insights into the geometry of the relevant eigenvarieties.

Moreover, by working with totally real fields of even degree we obtain isomorphisms between the relevant eigenvarieties and not just closed immersion, which is a stronger result than what is available over $\QQ$. In practice, this means that, in this case, we can study the full (cuspidal) Hilbert eigenvariety by using the quaternionic one.

\subsection*{Acknowledgements} The author would like to thank his supervisor Lassina Demb\'el\'e for his support and guidance. He would also like to thank Fabrizio Andreatta, David Hansen for useful comments and clarifications. This work is part of the authors thesis, so he wishes to thank the examiners Kevin Buzzard and David Loeffler for their careful reading and suggestions. Finally, he would like to thank the referee for the very useful feedback.\footnote{This work was partially supported by an EPSRC DTG at Warwick University and a EPSRC Doctoral Prize Fellowship at University College London, EP/N509577/1.}  
\counterwithin{thm}{subsection}

\begin{section}{\textbf{ Notation}}\label{notation}
\setcounter{subsection}{1}
	\begin{nota}\label{allnot}
		
		\begin{enumerate}[$(1)$]
			\item Let  $F$ be a totally real field of degree $g$, and fix a rational prime $p$ which is unramified with the exception of Sections $3$ and $4$ where it may be ramified unless otherwise stated.
			
			\item  Let $D$ be a quaternion algebra over $F$ with a fixed maximal order $\OO_D$ and let $G_D=\Res_{F/\QQ}(D^\times)$ (we will sometimes abuse notation and denote this simply by $D$). When $D=M_2(F)$ we denote this simply as $G$. We fix an isomorphism $\OO_D \otimes_{\OO_F} \OO_p \cong M_2(\OO_p)$, which  induces an isomorphism $D_p:=D \otimes_{F} F_p \cong M_2(F_p)$. Lastly, let $T$ denote a fixed maximal torus of $G_D$ and $\TT=\Res_{\OO_F/\ZZ}\mathbb{G}_m$.

			\item Let $\OO_F$ denote the ring of integers of $F$ and let $\d_F$ denote the different ideal of $F$. For each finite place $\vv$ of $F$ let $F_\vv$ denote the completion of $F$ with respect to $\vv$ and $\OO_\vv$ the ring of integers of $F_\vv$. For an integral ideal $\nn$, let $F_\nn= \prod_{\vv|\nn} F_\vv$ and similarly let $\OO_\nn=\prod_{\vv \mid \nn} \OO_{\vv}.$ In particular, if we have $p\OO_F=\prod_{i=1}^f \gothp_i$, then let $\OO_p=\oplus_i \OO_{\gothp_i}=\OO_F \otimes \ZZ_p$.

			\item Let $\Sigma$ to be the set of all places of $F$, $\Sigma_p$ be the set of all finite places above $p$ and $\Sigma_\infty \subset \Sigma$ the set of all infinite places of $F$.

			\item Let $\ov{\QQ}$ denote the algebraic closure of $\QQ$ inside $\CC$ and we fix an algebraic closure $\ov{\QQ}_p$ of $\QQ_p$. Furthermore, we fix  embeddings $inc:\ov{\QQ} \to \CC$ and $inc_p: \ov{\QQ} \to \ov{\QQ}_p$, which allow us to think of the elements of $\ov{\QQ}$ as both complex and $p$-adic numbers.

	\item
				Let $\pi_\ps$ denote the uniformisers  of $F_{\ps}$ and $\pi \in \OO_p$ be the element whose $\ps$ component of $\OO_p$ is $\pi_\ps$. By abuse of notation we also let $\pi$ denote the ideal of $\OO_F$ which is the product of all the primes ideals above $p$, i.e., the radical of $p\OO_F$.

			\item For each $\vv \in \Sigma_\infty$, we have a field embedding $\ii_\vv$ of $F$ into $\CC$ given by $\vv$. This map extends to a map $F_p \to \ov{\QQ}_p$ and then factors through the projection $F_p \to F_{\gothp}$ for some $\gothp$ above $p$. This then gives a natural surjection $\Sigma_\infty \to \Sigma_p$ where $\vv \mapsto \ps_\vv$.  For each prime ideal $\ps_j \in \Sigma_p$  let $\Si_{\ps_j}$ be the set of  $\vv \in \Sigma_\infty$ factoring through the projection $F_p \to F_{\ps_j}$.
			\item Let $L$ be a complete extension of $\QQ_p$, which contains the compositum of the images of $F$ under $\iota \circ \ii_\vv$, for $\vv \in \Sigma_{\infty}$ where $\ii: \CC \overset{\sim}\to \ov{\QQ}_p$ such that $\ii \circ inc= inc_p$.

		\end{enumerate}

	\end{nota}	
	
\end{section}

\begin{section}{\textbf{ Eigenvarieties and the Interpolation Theorem}}\label{eig}

We begin by briefly reviewing the construction of eigenvarieties and Chenevier's Interpolation Theorem. More details on general constructions can be found in many places, for example \cite{cmeig,chenjlc,buzeig,Hansen}. 

\subsection{\textbf{ The weight space}}\label{wei}
	
The weight space is a rigid analytic variety that allows us to make precise the idea of modular forms `living' in $p$-adic families. Following \cite{ME}, we begin with the classical definition of a weight of a Hilbert modular form.

\begin{defn}\label{cw}
	Let $n \in \ZZ_{\geq 0}^{\Sigma_\infty}$ and $v \in \ZZ^{\Sigma_\infty}$ such that $n+2v=(r,\dots, r)$ for some $r \in \ZZ$ and set $k= n+2$ and $w=v+n+1$. By abuse of notation we denote $(r,\dots,r)$ by $r$ for $r \in \ZZ$.  Note that all the entries of $k$ will have the same parity and $k=2w-r$. We call the pair $(k,r) \in \ZZ_{\geq 2}^{\Sigma_\infty} \times \ZZ$ a {\it classical algebraic weight}.   Note that given $k$ (with all entries paritious and greater than $2$) and $r$ we can recover $n,v,w$. In what follows we will move between both descriptions when convenient. We will call $(k,r,n,v,w)$ satisfying the above a {\it weight tuple}.
	
	A weight is called {\it arithmetic } or {\it classical} if it is the product of an algebraic character $(k,r)$ and a finite character $\psi$, which we denote by $(\k_{\psi},r)$.
		
	\end{defn}

\begin{defn}\label{wtspdf}\label{wtspstr}Let $\TT=\Res_{\OO_F/\ZZ} \mathbb{G}_m$. We define $\W^{G}$ to be the rigid analytic space over $L$ associated to the completed group algebra $\OO_L\llbracket\TT(\ZZ_p)\times \ZZ_p^\times\rrbracket$. We call $\W^{G}$ the {\it{weight space for}} $G$. Moreover, one has a universal character $$[-]:\TT(\ZZ_p) \times \ZZ_p^{\times} \lra \OO_L\llbracket\TT(\ZZ_p)\times \ZZ_p^\times\rrbracket^\times.$$
\end{defn}

	It follows from the above that $\W^G(\CC_p)=\Hom_{cts}(\TT(\ZZ_p)\times \ZZ_p^\times, \CC_p^\times)$. Moreover, 	we note that $\TT(\ZZ_p) \times \ZZ_p^{\times} \cong H \times \ZZ_p^{g+1}$, where $H$ is the torsion subgroup of $\TT(\ZZ_p) \times \ZZ_p^{\times}$. From this it follows that \[\W^{G} \cong H^{\vee} \times B(1,1)^{g+1} \cong \bigsqcup_{\chi \in H^\vee} \W_\chi\] as rigid spaces,  where $H^\vee$ is the character group of $H$ and $B(1,1)$ is the open ball of radius 1 around 1. It is clear from this that $\W^G$ is equidimensional of dimension $g+1$.

	\begin{nota}
	Elements of $\W^{G}(\CC_p)$ will be given by $v:\TT(\ZZ_p) \to \CC_p^{\times}$ and $r: \ZZ_p^{\times} \to \CC_p^{\times}$. Setting $n=-2v+r$ and $\k=n+2$, we will denote these weights as $(\k,r)$ and call $(\k,r,n,v,w)$ a weight tuple if $\k,r,n,v,w$ satisfy the same relations as in \ref{cw}. More generally, if $\V$ is an affinoid with a morphism of rigid spaces $\V \to \W^{G}$, then we will denote by $(\k^{\V},r^{\V})$ the restriction of the universal character to $\V$.
	\end{nota}

	\begin{defn}
		Let $(k,r,n,v,w)$ be a weight tuple with $(k,r) \in \ZZ^{\Sigma_\infty} \times \ZZ$ a classical algebraic weight. This defines an {\it algebraic weight}  by sending $(a,b) \in \TT(\ZZ_p)\times \ZZ_p^\times$ to $a^v b^r$.
	\end{defn}

	\begin{nota}\label{2.3.1}
		There is a natural map $\TT(\ZZ_p) \to \TT(\ZZ_p) \times \ZZ_p^{\times}$ given by $t \mapsto (t^{-2}, N_{F/\QQ}(t))$. In this way we view weights $(\k,r) \in \W^G$ with $(\k,r,n,v,w)$ a weight tuple as maps  $\TT(\ZZ_p) \to \CC_p$ given by $t \mapsto n(t)$.
	\end{nota}

	\begin{rmrk}\label{wtsp}
		In the literature there are slightly more general weight spaces than the one we have introduced. One alternative way of defining the weight space is to let $\W'$ denote the rigid analytic space associated to the completed group algebra  $\OO_L\llbracket T(\ZZ_p)\rrbracket$, where $T$ is a fixed maximal torus of $G$. The problem with this weight space is that it contains too many weights for which the associated spaces of modular forms would be empty. For this reason one usually imposes suitable vanishing conditions on these weights.  See \cite[Part III]{buzeig} and \cite[4.3.2]{urb}. The weight spaces one gets this way conjecturally have dimension $g+1$ (dependent on Leopoldt's conjecture).For this reason we have chosen to work with $\W^G$ which has the correct dimension.  Moreover, if Leopoldt's conjecture is true then the resulting eigenvarieties for the different weight spaces will be isomorphic.
		
	\end{rmrk}

Later, when defining the spaces of locally analytic functions it will be convenient for us to extend the  definition of the weight space from $\TT$ to $T$, which denotes a fixed maximal torus of $G$. We do this as follows:

\begin{defn}\label{maxtors}
	Let $(\k,r,n,v,w)$ be a weight tuple with $(\kappa,r) \in \W^G$ and set $$\l_{\kappa,r} \left ( \begin{matrix} a&0\\ 0&d \end{matrix} \right) =\l_1(a)\l_2(d)$$ where $\l_1=(r+n)/2$, $\l_2=(r-n)/2$. 
\end{defn}

\begin{rmrk}
	Note that if we map $T(\ZZ_p)$ to $\TT(\ZZ_p) \times \ZZ_p^\times$ via $\left ( \begin{smallmatrix}
	a & 0 \\ 0& d
	\end{smallmatrix} \right ) \mapsto (a/d, \text{Norm}(a))$ then our weights on $T$ and $\TT$ agree. 
\end{rmrk}
Using this, we talk about weights $\l$  on $T$ where we implicitly assume that there is some $(\kappa,r) \in \W^G$ such that $\l=\l_{\kappa,r}$. This construction then lets us take a weight $\W^G$ and get a weight in $\W'$ as in Remark \ref{wtsp}.

\begin{subsection}{Eigenvarieties}
	
	In order to define an eigenvariety $\mathscr{X}$, we need to specify the eigendata to which it is associated. We begin by recalling some standard definitions that can be found in \cite{buzeig,Hansen}.
 
 \begin{defn}Let $U$ be a compact operator on Banach space $M$. We define the \textit{Fredholm determinant} as \[\Fred_{M}(U)=\det(1-XU|M):=\lim_i \det(1-XU_i | M_i),\] where $(M_i)_i$ is a sequence of projective and finitely generated Banach modules such that $U_i:=U|_{M_i}$ converges to $U$ as $i \to \infty$.
 \end{defn}
 
For our purposes,  we will study the subspace of $\W^G \times \AA^1$ cut out by $f:=\Fred_{S^\dagger}(U_p)$, denoted $\CZ(U_p)$ or $\CZ(f)$, where $\W^G$ is the weight space as in Definition \ref{wtspdf} and $U_p$ is a compact operator on a space $S^\dagger$ of overconvergent Hilbert modular forms. These will give us our spectral varieties. 

\begin{defn} Let $ \V \subset \W^G$ be an affinoid and $\CZ(f)$ a spectral variety.  Define $\CZ_{\V,h}=\OOO(\V)\langle p^hX \rangle / (f(X)),$ which we view as an admissible affinoid open subset of $\CZ(f)$. We have a natural map $\CZ_{\V,h} \to \V$, which is flat but might not be finite. We say that $\CZ_{\V,h}$ is \textit{slope-adapted} if the above map is finite and flat.

\end{defn}

\begin{rmrk}
It is the possible to show that $\CZ_{\V,h}$ is slope adapted if and only if $f|_\V$ admits a slope $\leq h$ factorization $f|_\V(X)=Q(X)R(X)$, from which we have $\OOO(\CZ_{\V,h})=\OOO(\V)[X]/(Q(X))$. Then \cite[Theorem 4.6]{buzeig} tells us that the collection slope-adapted affinoids is an admissible cover of $\CZ(f)$. 
\end{rmrk}

\begin{defn}Let $A$ be commutative  Noetherian $K$-Banach algebra, for $K$ a field complete with respect to a non-trivial non-archimedean norm.  Following \cite{buzeig}, we say a Banach $A$-module $P$ satisfies property $(Pr)$ if there is a Banach $A$-module $Q$, such that $P \oplus Q$ (with its usual norm) is 
	potentially ON-able.
\end{defn}

\begin{defn}
	Let $M_1,M_2$ be Banach $R$-modules satisfying $(Pr)$ for $R$ a reduced affinoid and $\pmb{T}$ a commutative $R$-algebra with maps $\psi_i: \pmb{T} \to \End_R(M_i)$. Let $U \in \pmb{T}$ act compactly on both $M_1$ and $M_2$. A continuous $R$-module and $\pmb{T}$-module homomorphism $\a: M_1 \to M_2$ is called a {\it primitive link} if there is a compact $R$-linear and $\pmb{T}$-linear map $c: M_2 \to M_1$ such that $\psi_2(U): M_2 \to M_2$ is $\a \circ c$ and $\psi_{1}(U): M_1 \to M_1$ is $c \circ \a$. More generally a continuous $R$-module and $\pmb{T}$-module homomorphism $\a: M' \to M$ is a {\it link} if there exists a sequence $M_i$ of Banach $R$-modules satisfying $(Pr)$ for $i \in \{0,\dots,n\}$ such that $M'=M_0$, $M=M_n$ and $\a$ factors as a compositum of maps $\a_i:M_i \to M_{i+1}$ with $\a_i$ a primitive link.
\end{defn}

\begin{defn}\label{eigdat}
	Let $\W$ be a reduced rigid space, $R$ a reduced affinoid and $\pmb{T}$ be a commutative $R$-algebra with a specified element $U$. For admissible affinoid open $\V \subset \W$ let $M(\V)$  a  Banach $\OOO(\V)$-module  satisfying $(Pr)$ with an $R$-module homomorphism $\psi_\V: \pmb{T} \to \End_{\OOO(\V)}(M_\V)$ such that $\psi_\V(U)$ is compact. Finally assume that if $\V \subset \V' \in \W$ are two admissible affinoid opens, then there is a continuous $\OOO(\V)$-module homomorphism $\a: M_{\V} \to M_{\V'} \hat{\otimes}_{\OOO(\V')} \OOO(\V)$ which is a link and such that if $\V_1 \subset \V_2 \subset \V_2 \subset \W$ are all affinoid subdomains then $\a_{13}=\a_{23} \circ \a_{12}$ for $\a_{ij}: M_{\V_i}  \to M_{\V_i} \widehat{\otimes}_{\OOO(\V_i)} \OOO(\V_j)$.
	
	We give the name of {\it eigendata} or {\it eigenvariety data}, to tuple $\mathfrak{E}=(\W,\CM,\pmb{T},U)$ where $\CM$ is the coherent sheaf defined by the $M_\V$.

\end{defn}	

\begin{thm}[\bf The Eigenmachine]\label{th1} Attached to $\mathfrak{E}=(\W,\CM,\pmb{T},U)$ there is a canonically associated rigid space $\CX(\mathfrak{E})$ with a finite morphism to the spectral variety $\CZ(U)$ defined by $U$ and whose points over $z \in \CZ(U)$ are in bijection with the generalized eigenspace for the action of\/ $\pmb{T}$ on the fibre $\CM_z$. Moreover, if\/ $\W$ is equidimensional of dimension $n$, then so is $\CX(\mathfrak{E})$.
	
\end{thm}

\begin{proof}
	This follows from \cite[Construction 5.7, Lemmas 5.8-5.9]{buzeig}.
\end{proof}

\end{subsection}

\begin{subsection}{The Interpolation Theorem}

In this section we will recall Chenevier's interpolation theorem, which we will use in order to interpolate the Jacquet-Langlands correspondence. For this we will need to find a very Zariski dense subset of the weight space, together with a classical structure on it. As the name suggests, the classical structure will be given by  the subspace of classical modular forms inside the space of overconvergent modular forms. We then use this to find closed immersions between different eigenvarieties by relating their classical structures.  In our case, it will be the classical Jacquet-Langlands correspondence that will allow us to relate the classical structures. Following \cite{chenjlc}, we have:
\begin{defn}
	A subset $X \subset Z$ is Zariski dense in $Z$ if for every analytic subset (see \cite[Section 9.5.2]{bgr}) $Y \subset Z$ such that $X \subset Y$, then $Y=Z$.
\end{defn}

\begin{defn} A Zariski dense subset $X \subset \W^G(\CC_p)$ is \textit{{very Zariski dense}} if for each $x \in X$ and each irreducible admissible affinoid open $V \subset \W^G$ containing $x$, we have that $V(\CC_p) \cap X$ is Zariski dense in $X$.

\end{defn}

From this we define the classical structures as follows: 
	\begin{defn}\label{51} Let $\mathfrak{E}=(\W,\CM,\pmb{T},U)$ be a set of eigendata as above and let $X \subset \W$ be a very Zariski dense subset. For each $x \in X$, let $\CM^{cl}_x$ be a  finite dimensional $\pmb{T}$-module contained in $\CM_x$ and, for every $h \in \RR$, set $X_{h}=\{x \in X \mid \CM^{\leq h}_x \subset \CM^{cl}_x \}.$ We say that $\CM^{cl}$ gives a {\it classical structure} on $X$ if for every open affinoid neighbourhood $V \subset \W$ and every $h$, the sets $X\cap V, X_h \cap V$ have the same Zariski closure in $V$.
	
\end{defn}

\begin{defn}\label{Red} If $\CX$ is an eigenvariety, with eigendata $\mathfrak{E}=(\W,\CM,\pmb{T},U)$, we denote the {\it{nilreduction}} of $\CX$ by $\CX^{red}$, and we say that an eigenvariety is {\it{reduced}} if $\CX^{red} \cong \CX$. 
\end{defn}

With these definitions we can now state the Chenevier's {\it Interpolation theorem.}

\begin{thm}(Chenevier)\label{int} Let\/ $\CX_i$ be eigenvarieties associated to the eigendata of $\mathfrak{E}_i=(\W_i,\CM_i,\pmb{T}_i,\psi_i),$ for $i=1,2$\/ with  $\W=\W_1=\W_2$ and\/ $\pmb{T}=\pmb{T}_1=\pmb{T}_2$.
	Let\/ $X \subset \W$ a  very Zariski dense subset such that \/ $\CM_i^{cl}$ is a classical structure on $X$ for each\/ $\CM_i$. Assume that, for all\/ $t \in \pmb{T}$ and all\/ $x \in X$, we have 
	$$\det \left(1-\psi_1(tU)Y\mid_{ \CM_{1,x}^{cl}} \right ) \text{ divides } \det\left (1-\psi_2(tU)Y\mid_{\CM_{2,x}^{cl}}\right)$$  in\/ $k(x)[Y]$, where\/ $k(x)$ is the residue field a\/t $x$. Then, there is a canonical closed immersion\/ $\iota: \CX_1^{red} \hookrightarrow \CX_2^{red}$ such that the following diagrams commute

	\begin{center}
		\begin{tikzcd}
		\CX_1^{red} \arrow[r, hook,"\iota"] \arrow[d] & \CX_2^{red} \arrow[dl] &  & 	\pmb{T}  \arrow[r, "\phi_1^{red}"]  \arrow[d,"\phi_2^{red}"] & \OOO(\CX_1^{red})   \\
		\W & && \OOO(\CX_2^{red}) \arrow[ur, "\iota^*"] & 
		\end{tikzcd}
	\end{center}

\end{thm}

\begin{proof}See \cite[Theorem 1]{chenjlc}.
	
\end{proof}

	\begin{cor}\label{coro}
	If\/  $\det\left (1-\psi_1(tU)Y_{\mid \CM_{1,x}^{cl}} \right ) = \det \left (1-\psi_2(tU)Y_{\mid \CM_{2,x}^{cl}}\right )$  in\/ $k(x)[Y]$ for all\/ $t \in \pmb{T}$ and all\/ $x \in X$, then there is an isomorphism\/ $\CX_1^{red} \cong \CX_2^{red}$.
\end{cor}

\begin{proof}In this case the above Theorem gives us two closed immersions $\iota_{12}: \CX_1^{red} \hookrightarrow \CX_2^{red}$ and $\iota_{21}: \CX_2^{red} \hookrightarrow \CX_1^{red}$, from which the result follows at once by noting that $\iota_{12}\iota_{21}=Id_{\CX_2}$ and $\iota_{21}\iota_{12}=Id_{\CX_1}$. 
	
\end{proof}

We now have a result of Chenevier that gives a criterion for an eigenvariety to be reduced. Suppose that\/ $X \subset \W$ is a very Zariski dense subset  giving  a classical structure. For\/ $h \in \RR$, let $X_{h}^{ss} = \{ x \in X  \mid \CM^{cl}_x \cap \CM^{\leq h}_x \text{ is a semisimple } \pmb{T}\text{-module} \}.$

\begin{lem}\label{red}
	If for all\/ $h \in \RR$, $x \in X$ and $V \subset \W$  an open affinoid containing\/ $x$, there exists $W \subset V$ an open affinoid containing\/ $x$, such that\/ $X_{h}^{ss} \cap W$ contains an open Zariski dense subset of\/ $X \cap W$, then\/ $\CX$ is reduced (here we view\/ $X \cap W$ as a topological subspace of\/ $W$ with the Zariski topology).

\end{lem}

\begin{proof} See \cite[Proposition 3.9]{chenjlc}.
\end{proof}

\end{subsection}

\end{section}

\begin{section}{\textbf{ Eigenvariety of Hilbert modular cusp forms}}
\setcounter{subsection}{1}
In this section  we define the spaces of Hilbert modular forms, in a way which generalizes to give the spaces of overconvergent Hilbert modular cusp forms. We will only discuss the classical spaces here and give some indication of what is needed to extend this to the overconvergent setting as is done in \cite{AIP}. Throughout this section our chosen prime $p$  may be ramified unless otherwise stated.

To define the spaces of Hilbert modular forms for $G$, we first work with the group $G^*=G \times_{\Res_{F/\QQ} \mathbb{G}_m}\mathbb{G}_m$ and define the spaces of modular forms for $G^*$. Then using a projector one gets the definition for $G$. The reason for working with $G^*$ is that the relevant moduli problem associated to $G^*$ is representable while the one for $G$ is not.

 Following Hida \cite{hidapadic}, we consider the fibered category $\A_F$ of abelian schemes over the category of schemes.

\begin{defn}
 An object of $\A_F$ is a triple $(A,\l,\iota)$ where: 

\begin{enumerate}[$(1)$]
\item $A$ is an abelian scheme $A \to S$ of relative dimension $g$ over $S$.
\item $\OO_F$-multiplication given by an embedding $\iota:\OO_F \hookrightarrow \End_S(A)$.
\item  Let $\gothc$ be a fractional ideal of $F$ and $\gothc^+$ its cone of totally positive elements. If $P \in \Hom_{\OO_F}(A,A^{\vee})$ is a sheaf for the \'{e}tale topology on $S$ of symmetric $\OO_F$-linear homomorphisms from $A$ to the dual abelian scheme $A^{\vee}$, and if $P^{+} \subset P$ is the subset of polarizations, then \[\l:(P,P^{+}) \lra (\gothc,\gothc^{+})\] is an isomorphism of \'{e}tale sheaves, as invertible $\OO_F$-modules with a notion of positivity. The triple $(A,A^{\vee},\gothc)$ is subject to the Deligne-Pappas conditions, which means the map $A \otimes_{\OO_F} \gothc \to A^{\vee}$ is an isomorphism of abelian schemes.
\end{enumerate}
See \cite[Section 4.1]{hidapadic} and \cite{AIP} for more details.

\end{defn}

We now state the moduli problem for $G^*$ associated to the $\mu_\nn$-level structure on an abelian variety.

 \begin{defn}\label{modprob}
 	
 	Let  $(A,\iota, \l)$ in $\A_F$ (so $\l$ is a $\gothc$-polarization). Let $\nn$ be a non-zero ideal and let $\mu_\nn$ denote the locally free group scheme of finite rank given by $\mu_\nn(R)=\{x \in \mathbb{G}_m(R) \otimes_{\ZZ} \d_F^{-1} | \nn x=0\}$. Let $\nn \cap \ZZ=(N)$ and let $M(\gothc, \mu_\nn)$ be the Hilbert moduli scheme representing the functor $\E_{\mu_\nn}: \Sch/_{\ZZ[1/N]}  \to \Set$ where $\E_{\mu_\nn}(S)$ is the set of isomorphism classes of $(A_{/S},\iota,\l,\Phi_\nn)$. Here  $\Phi_\nn:\mu_\nn  \hookrightarrow A[N]_{/S}$ is a closed immersion compatible with $\OO_F$-actions. We call such a $\Phi_\nn$ a $\mu_\nn$-{\it level structure} on $A$.\footnote{If we take $\mu_{(N)}$ with $N \geq 3$, then the associated moduli problem is representable by a scheme $M(\gothc,\mu_{(N)})$ (cf. \cite[Chapter 3, Theorem 6.9]{gor}).}
 	
 \end{defn}

\begin{rmrk}
	For a fractional ideal $\gothb \in \OO_F$ let $\gothb^*=\gothb^{-1}\d_F^{-1}$. Let\/ 
	\begin{align*}
	&\Gamma_0(\gothb,\nn)=	
\GL_2(F)^+ \cap   \left ( \begin{smallmatrix}
\OO_F &  \gothb^{*} \\ \gothb\nn\d_F &\OO_F 
\end{smallmatrix}\right )\\
&\Gamma_1(\gothb,\nn) =\left\{   \left ( \begin{smallmatrix}
a & b \\ c &d 
\end{smallmatrix}\right ) \in \Gamma_0(\gothb,\nn) \mid d \equiv 1 \mod \nn \right\}\\
&\Gamma_1^1(\gothb,\nn)=\Gamma_1^1(\gothb,\nn) \cap \SL_2(F).
	\end{align*}	
Then one can show that, for $\nn$ with $(N) =\nn \cap \ZZ$ such that $N \geq 3$ ,    $M(\gothc,\mu_\nn)(\CC)=\Gamma_1^1(\gothc,\nn) \backslash \HHH^g$.
	
\end{rmrk}

\begin{nota}
	Let $M(\gothc, \mu_\nn)$ denote the scheme representing the corresponding moduli problem (for $G^*$).  Denote by $\ov{M}(\gothc, \mu_\nn)$ (resp. $\ov{M}^*(\gothc, \mu_\nn)$) {\it a} fixed {\it toroidal}  (resp. {\it the minimal}) compactification of $M(\gothc, \mu_\nn)$. 

\end{nota}

\begin{defn}
	Let $\pmb{A}$ be a versal semiabelian scheme over  $\ov{M}(\gothc, \mu_\nn)$ , with real multiplication by $\OO_F$. Define  $\w_{\pmb{A}}$ to be the conormal sheaf to the identity of $\pmb{A}$.
\end{defn}

There exist a greatest open subscheme  $\ov{M}^{R}(\gothc,\mu_\nn) \subset  \ov{M}(\gothc,\mu_\nn)$ such that $\w_{\pmb{A}}$ is an invertible $\OO_{ \ov{M}^{R}(\gothc,\mu_\nn)} \otimes_{\ZZ} \OO_F$-module. We will use $\w_{\pmb{A}}$ to define {\it another} invertible sheaf whose sections will be our  Hilbert modular forms. But in order to do so, we first need to define the weight space for  $G^*$ and showing how it is related to $\W^G$.

\begin{defn}\label{gstar}
The weight space $\W^{G^*}$ for $G^*$ is defined by setting $\W^{G^*}$ to be the rigid analytic space over $L$ associated to $\OO_L\llbracket \TT(\ZZ_p) \rrbracket$, where $\TT$ and $L$ are as before. We will denote such weights by $\k$ or $k$.

\end{defn}

There is a canonical map $\W^G \to \W^{G^*}$ induced from \ref{2.3.1}. 	We define a classical algebraic weight for $G^*$ as a map from $\TT(\ZZ_p)$ to $\CC_p$ defined by an element $k \in \ZZ^g_{\geq 0}$, as usual.

We now construct an invertible sheaf associated to classical weights in $\W^{G^{*}}$, from which we can then define the spaces of Hilbert modular forms for $G^{*}$:
\begin{defn}
Recall that $L$ is a splitting field for $F$, and  let $k \in \ZZ^{\Sigma_\infty}$ be a classical weight  for $G^*$. Then, define the invertible modular sheaf $$\Omega^{k}:=\bigotimes_{\v \in \Sigma_\infty} \w_{\pmb{A},\v}^{\otimes k_{\v}},$$ where $\w_{\pmb{A},\v}:=\w_{\pmb{A}} \otimes_\v \OO_L$ and here $\otimes_\v$ denotes the tensor over $\OO_L \otimes \OO_F$ via $1 \otimes \iota_\v$. 

\end{defn}

\begin{defn}
	The $L$-vector space of $\gothc$-polarized, tame level $\Gamma_1^1(\gothc,\nn)$ and weight $k$ {\it Hilbert modular forms for $G^*$} is defined by
$$M_k(\Gamma_1^1(\gothc,\nn)):=H^{0}( \ov{M}^{R}(\gothc, \mu_\nn),\Omega^{k}).$$ 
The subspace of cusp forms is defined by $$S_{k}(\Gamma_1^1(\gothc,\nn)):=H^0(\ov{M}^R(\gothc, \mu_\nn),\Omega^{k}(-B)),$$
where $B:=\ov{M}(\gothc, \mu_\nn) \backslash M(\gothc, \mu_\nn)$ is the boundary divisor in the toroidal compactification.  
\end{defn}

To define the spaces of  Hilbert modular forms associated to $G$, we need to introduce a certain projector. First, we note that multiplication by $\e \in \OO_F^{\times,+}$ gives an isomorphism $(A,\iota,\l,\Phi) \cong (A,\iota,\e^2 \l,\e \Phi)$. Now, let  $\mathfrak{S}_\nn$ be the elements of $\OO_F^{\times,+}$ congruent to $1$ modulo $\nn$.   Define an action of $\mathfrak{O}:=\OO_F^{\times,+}/\mathfrak{S}_{\nn}^2$ on $M(\gothc, \mu_\nn)$, by $\e \cdot (A,\iota,\l,\Phi):=(A,\iota,\e\l,\Phi).$ Since multiplication by $\e$ gives an isomorphism $\e:A \to A$ such that $\e^*\l=\e^2\l$, it  follows that, if $\e=\eta^2 \in \mathfrak{S}_{\nn}$, then $\e$ acts trivially on $(A,\iota,\l,\Phi)$; hence the action factors through $\mathfrak{O}$ as required.

\begin{defn}
	Let $(k,r,n,v,w)$ be a weight tuple with $(k,r)$ a classical algebraic weight (for $G$). We define an action of $\mathfrak{O}$ on $\Omega^{k}$ by sending a local section $f$ of $\Omega^{k}$ on $M^R(\gothc, \mu_\nn)$ to $$(\e\cdot f):(A,\iota,\l,\Phi, \beta) \to w(\e)f(A,\iota,\e^{-1}\l,\Phi,\beta) $$ where $\e \in \OO_F^{\times,+}$ and $\beta$ is a local generator for $\w_{\pmb{A}}$ as a $\OO_{M^R(\gothc, \mu_\nn)} \otimes \OO_F$-module. If $\e=\eta^2 \in \mathfrak{S}_\nn$, then this acts trivially. Hence the action factors through $\mathfrak{O}$. With this we define a {\it projector} $\mathfrak{e}_{k,r}:M_{k}(\Gamma_1^1(\gothc,\nn)) \to M_{k}(\Gamma_1^1(\gothc,\nn))$ by $$\mathfrak{e}_{k,r}:=\frac{1}{\mid\mathfrak{O}\mid} \sum_{\e \in \mathfrak{O}} \e.$$ 
\end{defn}

\begin{defn}  The $L$-vector space of {\it classical Hilbert modular forms for $G$ of level $\Gamma_1(\gothc,\nn)$,  and weight $(k,r)$} is defined to be the image of $\mathfrak{e}_{k,r}$ and is denoted $M_{k,r}^G(\Gamma_1(\gothc,\nn))$. Similarly, we let $S_{k,r}^G(\Gamma_1(\gothc,\nn))$ be the image of $S_{k}(\Gamma_1^1(\gothc,\nn))$ under $\mathfrak{e}_{k,r}$. 
\end{defn}

\begin{rmrk}\label{classicalex}
	We note here that there are other ways of defining Hilbert modular forms for $G$ as sections of a sheaf $\Omega^{(k,r)}$ on $M^G(\gothc,\mu_\nn)$ (cf. \cite[Section 2.2]{tian}). Working over $\CC$ one then recovers the spaces $S_{k,r}(\Gamma_1(\gothc,\nn))$ as defined above. To see the relationship with our definition one observes that there is a morphism $m: M^R(\gothc,\mu_\nn) \to M^G(\gothc,\mu_\nn)$ which is finite and Galois, with Galois group $\mathfrak{D}$ such that $\Omega^{k,r}=(m_*(\Omega^{k}))^\mathfrak{D}$ (cf. \cite[Section 1]{AIP}).
\end{rmrk}

Note that these spaces will not be fixed by the Hecke operators. In fact, note that  $F^{\times,+}$ acts on the pairs $(\gothc,\gothc^+)$ by $\e(\gothc,\gothc^+)=(\e\gothc,\e\gothc^+)$, which induces an isomorphism $\a_\e: M_{k,r}^G(\Gamma_1(\gothc,\nn)) \to M_{k,r}^G(\Gamma_1(\e \gothc,\nn))$. Moreover, if $\e \in \OO_F^{\times,+}$, then $\a_\e(f)=f$ for all $f \in M_{k,r}^G(\Gamma_1(\gothc,\nn))$. To fix this, we must work adelically. Let $U_1^G(\nn)= \left \{\g \in G(\widehat{\ZZ}) \mid \g \equiv \left ( \begin{smallmatrix} *&*\\ 0&1  \end{smallmatrix} \right) \mod \nn \right \}$. We can now define the space of Hilbert modular forms for $G$.

\begin{defn}\label{HMFSdefn}	
	We define the space of {\it classical Hilbert modular forms for $G$ of level $U_1(\nn)$ and weight $(k,r)$} denoted $M_{k,r}^G(U_1^G(\nn))$ as $V/I$ where $$V:=\underset{(\gothc,\gothc^+)}\bigoplus M_{k,r}^G(\Gamma_1(\gothc,\nn))$$ and $I=(f-\a_\e(f))_{\e \in \left ( F^{\times,+}/\OO_F^{\times,+}\right )}$. We define $S_{k,r}^G(U_1^G(\nn))$ similarly. 
\end{defn}

One can define Hecke operators on these spaces, which satisfy the all the usual properties. We denote by $\pmb{T}(U)$ the Hecke algebra consisting of Hecke operators away from the level together with $U_{\ps}$ for $\ps\in  \Sigma_p$. Moreover, when we come to the Jacquet-Langlands correspondence, we will assume that $\pmb{T}(U)$ does not contain Hecke operators at the finite places dividing the discriminant of the relevant quaternion algebra (if there are any).

\subsection{\textbf{ Overconvergent spaces}}
Our goal is now to associate to each weight $(\k,r)$ or family of weights $(\k^{\V},r^\V)$, a space of overconvergent Hilbert modular forms. There are several constructions of these spaces but we will be interested in the construction given by \cite{AIP}. In this case, one defines an overconvergent sheaf which interpolates the sheaf $\Omega^{k}$. Using this, one can then define the spaces of overconvergent Hilbert modular forms for $G^*$ and then, using a projector, define the spaces for $G$. The construction of the overconvergent sheaf can be found in \cite{AIP,AIP4,AIS2,hattori}, so we only give some of its properties.

Let $f$ be the number of primes above $p$ in $F$ and let $t_m \in \QQ^{f}$ be a multi-index with $0< t_i \leq \frac{1}{p^m}$ for $m \geq 1$. Let $\ov{\MM}(\gothc, \mu_\nn)$ and $\ov{\MM}^*(\gothc, \mu_\nn)$ denote the formal completions of $\ov{M}(\gothc, \mu_\nn),\ov{M}^*(\gothc, \mu_\nn)$  along their special fibres. Now, let $\ov{\M}(\gothc, \mu_\nn), \ov{\M}^*(\gothc, \mu_\nn)$ denote the rigid fibres of $\ov{\MM}(\gothc, \mu_\nn), \ov{\MM}^*(\gothc, \mu_\nn)$ respectively  and let $\ov{\M}(\gothc, \mu_\nn,t_m), \ov{\M}^*(\gothc, \mu_\nn,t_m)$ denote the neighbourhoods of the respective ordinary locus defined by the condition that ${\val}_p(h_{\gothp_i}) \cap [0,1] \leq t_i$ (this is the truncated valuation), where $h_{\gothp_i}$ are the {\it partial Hasse invariants} as defined in \cite[3.2.1]{AIP}.

Before continuing let us recall what it means for a weight to be $m$-analytic, as in \cite{AIP}:
\begin{defn}\label{ann}
	Let $\wh{\TT}$ be the formal group obtained by completing $\TT$ along the unit section and let $\TT_s^0$ be the formal subgroup of $\wh{\TT}$ of elements which are congruent to $1 \mod p^s$. We say a weight $\k \in \W^{G^*}$ is $m$-analytic for some $m \in \ZZ_{\geq 0}$ if its restriction to $\TT_m^0(\ZZ_p)$ factors as $\exp \circ \psi \circ \log$ where $\psi$ is some $\ZZ_p$-linear map and $\exp,\log$ are the $p$-adic exponential and $p$-adic logarithm. We say a weight tuple is $m$-analytic if it corresponds to an $m$-analytic weight under the map $\W^G \to \W^{G^*}$ from \ref{gstar}.
\end{defn}

Now, the overconvergent sheaves are defined over formal models of $\ov{\M}(\gothc, \mu_\nn,t_m)$ and $\ov{\M}^*(\gothc, \mu_\nn,t_m)$, which are obtained as follows: let $\ov{\MM}(\gothc, \mu_\nn,t_m))$ (resp. $\ov{\MM}^*(\gothc, \mu_\nn,t_m))$ be the normalization (in $\ov{\M}(\gothc, \mu_\nn,t_m)$ (resp. $\ov{\M}^*(\gothc, \mu_\nn,t_m)$)  of the formal model of $\ov{\M}(\gothc, \mu_\nn,t_m)$ (resp. $\ov{\M}^*(\gothc, \mu_\nn,t_m)$ ) given by taking iterated blow-ups along the ideals $(h_{\gothp_i},p^{t_i})$ of $\ov{\MM}(\gothc, \mu_\nn)$ (resp. $\ov{\MM}^*(\gothc, \mu_\nn)$ ) and removing all divisors at infinity. Then on  $\ov{\MM}(\gothc, \mu_\nn,t_m))$ we can construct the following sheaves:

\begin{thm}[Andreatta--Iovita--Pilloni]For every\/ $m$-analytic weight\/ $\k \in \W^{G^*}(L)$ there exists a coherent sheaf\/ $\Omega^{\dagger,\k}$ of\/ $\OO_{\ov{\mathfrak{M}}(\gothc, \mu_\nn,t_m)}$-modules whose restriction to the rigid analytic fibre\/ $\ov{\mathcal{M}}(\gothc, \mu_\nn,t_m)$ is invertible. 
	
	More generally, to each affinoid\/ $\V$ with a morphism\/ $\V \to \W^{G^*}$ and\/ $m$ such that\/ $\k^{\V}$ is locally\/ $m$-analytic, one can attach a coherent sheaf\/ $\Omega^{\dagger,\k^{\V}}$ of\/ $\OO_{\ov{\mathfrak{M}}(\gothc, \mu_\nn,t_m) \times \hat{\V}}$-modules where\/ $\hat{\V}=\Spf(A)$ is the formal model of\/ $\V$, where\/ $A$ consists of power bounded elements of $\V$. Moreover, the restriction of\/ $\Omega^{\dagger,\k^{\V}}$  to the rigid fibre is invertible. Lastly, if\/ $k$ is a classical weight (for $G^*$), then\/ $\Omega^{\dagger,k}$ agrees on\/ $\ov{\mathcal{M}}(\gothc, \mu_\nn,t_m)$ with the classical\/ $\Omega^{r}$.
	
\end{thm}
\begin{proof}
	See  \cite[Section 3.4-3.5]{AIP}.
\end{proof}

\begin{rmrk}
	In general, the $\Omega^{\dagger,\k^{\V}}$ and $\Omega^{\dagger,\k}$  depend on $m$, but when restricted to the rigid fibres, they are independent of $m$, for this reason we have suppressed the dependence on $m$. See \cite[Proposition 3.9 and Proposition 3.13]{AIP}.
\end{rmrk}

Using these sheaves, one can then define the spaces of $t_m$-overconvergent cuspidal Hilbert modular forms for $G^*$ of weight $\k^\V$ by setting $$S_{\V}^{\dagger}(\Gamma_1^1(\gothc,\nn), t_m)=\hH^0(\ov{\M}(\gothc, \mu_\nn,t_m)\times \V, \Omega^{\dagger,\k^{\V}}(-B))$$ where $B$ is again the boundary divisor. From this, one then uses a projector $\mathfrak{e}_{\k^\V,r^\V}$ to define families of $t_m$-overconvergent cuspidal Hilbert modular forms for $G$ of weight $(\k^\V,r^\V)$ denoted $S_{\V}^{G,\dagger}(\Gamma_1(\gothc,\nn), t_m)$. Moreover, taking $\V=\Spf(L)$ gives $S_{\k,r}^{G,\dagger}(\Gamma_1(\gothc,\nn),t_m)$.

\begin{thm}
Let\/ $\V$ be an admissible open affinoid of\/ $\W^G$ and\/ $(\k^\V,r^\V)$ the restriction of the universal character to $\V$. Let\/ $A$ be the algebra of power bounded elements of $\V$. Then for an appropriate\footnote{This means such that\/ $\k^\V$ is\/ $m$-analytic (see Definition \ref{ann}).} choice of\/ $m$ and\/ $t_m$ the spaces\/ $S_{\V}^{G,\dagger}(\Gamma_1(\gothc,\nn),t_m)$  are  Banach\/ $(A \otimes_{\OO_L} L)$-modules satisfying $(Pr)$. Moreover, for any weight\/ $(\k,r) \in \V(L)$ there is a natural specialization map $$S_{\V}^{G,\dagger}(\Gamma_1(\gothc,\nn),t_m) \lra  S_{\k,r}^{G,\dagger}(\Gamma_1(\gothc,\nn),t_m)$$ which is surjective.

\end{thm}

\begin{proof}
This is \cite[Theorem 4.4]{AIP}.
\end{proof}

Now as before these spaces have an action of $F^{\times,+}$ so they will not be fixed under the action of Hecke operators. In particular, we have:

\begin{lem}Let\/ $\e \in F^{\times,+}$ and assume that\/ $\e$ is also a\/ $p$-adic unit. Then there is a canonical isomorphism $$\mathscr{L}_\e: S_\V^{G,\dagger}(\Gamma_1(\gothc,\nn),\un{t}) \lra S_\V^{G,\dagger}(\Gamma_1(\e \gothc,\nn),\un{t})$$ which only depends on \/ $\e$ modulo totally positive units.
	
\end{lem}

\begin{proof}
	This is \cite[Lemma 4.5]{AIP}.
\end{proof}

\begin{defn} Let $$S_\V^{G,\dagger}(U_1(\nn),\un{t}):= \left ( \bigoplus_{\gothc \in Frac(F)^{(p)}} S_\V^{G,\dagger}(\Gamma_1(\gothc,\nn),\un{t})\right ) / \left (\mathscr{L}_\e(f)-f \right)_{\e \in Princ(F)^{+,(p)}}$$ be the Banach module of tame level $U$, {\it $\un{t}$-overconvergent   cuspidal arithmetic Hilbert modular forms} for $G$ with weights parametrized by $\V$. Here $Frac(F)^{(p)}$ is the group of fractional ideals prime to $p$ and $Princ(F)^{+,(p)}$ is the group of positive elements which are $p$-adic units.
	
	Moreover, taking the limit over $\un{t}$ we get Frechet spaces  $S_\V^{G,\dagger}(U_1(\nn))$ which give a quasi-coherent sheaf of overconvergent cuspidal arithmetic Hilbert modular forms $S^{G,\dagger}(U_1(\nn))$ over $\W^G$, whose value at an open affinoid $\V \subset \W^G$ is $S_\V^{G,\dagger}(U_1(\nn))$. 
\end{defn}

\begin{rmrk}
	Note that taking $\V=\Sp(L)$ with image $(\k,r)$ in $\W^G$ will give the spaces of this fixed weight. 
\end{rmrk}

Following \cite[Section 4.3]{AIP}, for $\gothq$ prime to the tame level, one can define commuting Hecke operators $T_\gothq,S_\gothq$ action on $S^{G,\dagger}(U_1(\nn))$. Moreover, for $\gothp|p$ one can define operators $U_{\gothp}$ such that $U_p=\prod_{\gothp|p} U_{\gothp}^{e_\gothp}$ for $e_\gothp$ the ramification degree of $\gothp$.

\begin{prop}
	The $U_p$ operator is a compact operator on $S_{\k,r}^{G,\dagger}(U_1(\nn))$ for any weight $(\k,r)$. 
\end{prop}

\begin{proof}
	This follows from \cite[Lemma 3.27]{AIP}.
\end{proof}

\begin{defn}\label{slope}
	Let $h \in \QQ_{\geq 0}$. We say an element $f \in  S_{\k,r}^{G,\dagger}(U_1(\nn))$ has {\it slope}-$\leq h$ for $U_p$ (resp. $U_{\gothp}$ for $\gothp | p$) if it is annihilated by a unitary polynomial in $U_p$ (resp. $U_{\gothp}$) whose roots have valuation less than $h$.\footnote{	Note that if $f$ is in fact an eigenform, then having slope-$\leq h$ for $U_p$ (resp. $U_{\gothp}$) is saying that the $p$-adic valuation of the $U_p$ (resp. $U_{\gothp}$) eigenvalue is less than $h$.}
\end{defn}

We now wish to use Buzzard's Eigenmachine to construct the eigenvariety of Hilbert modular forms. One of the key ingredients is the existence of links which is checked explicitly in \cite[Section 3.3.3]{hattori}.

\begin{thm}\label{AIPeig} 
	Associated to the eigendata of $(\W^G,S^{G,\dagger}(U_1(\nn)),\pmb{T},U_p)$ we have an eigenvariety $\CX_G(U_1(\nn))$ with the following properties:
	
	\begin{enumerate}[$(a)$]
		\item It is equidimensional of dimension $g+1$.
		\item There is a universal character $\phi:\pmb{T} \to \OOO_{\CX}$.
		\item There is a map $\a:\CX \to \W^G$ that is locally on $\CX$ and $\W^G$, finite and surjective.
		\item For all $(\kappa,r) \in \W^G$, the points $\a^{-1}(\k,r)$ are in bijection with the finite slope eigensystems occurring in $S^{G,\dagger}(U_1(\nn))\mid_{\kappa,r}=S^{G, \dagger}_{\kappa,r}(U_1(\nn))$.
		\item Let\/ $(k,r)$ be a classical weight in\/ $\W^G$. Let\/ $f \in S_{k,r}^{G,\dagger}(U)$ be a finite slope (for\/ $U_p$) overconvergent Hilbert modular form whose\/ $U_{\gothp_i}$ slope is less than\/ $h_i$ for\/ $\gothp_i \in \Si_p$. If\/ $p$ is unramified and\/ $h_i <  v_{\ps_i}(k,r) +\min_{j \in \Si_{\gothp_i}} \{k_j-1\} $ for all\/ $i$,  then\/ $f$ is a classical form.
		
	\end{enumerate}

\end{thm}
\begin{proof}
This is essentially \cite[Theorem 5.1]{AIP}, the only modification is that using a stronger version of the control theorem due to \cite{tian}, which matches up with the one we will have on the quaternionic side.
\end{proof}

\end{section}

\begin{section}{\textbf{ Eigenvarieties for totally definite quaternion algebras}}\label{totdef}

 Following  \cite[Part III]{buzeig}, we will define the eigenvariety attached to a totally definite quaternion algebra $D$ over $F$ and prove the control theorem in this setting. In contrast to \cite{buzeig}, we will work with the weight space $\W^G$ which, as we explained above, has the advantage of being equidimensional of dimension $g + 1$ (recall $[F:\QQ]=g$). Apart from this small detail, the rest of our construction spaces of overconvergent quaternionic modular forms over $F$ follows~\cite[Part III]{buzeig}. Throughout this section our chosen prime $p$ may be ramified unless otherwise stated. 

\begin{subsection}{Classical spaces}\label{quat}
	
We will define the spaces of classical quaternionic modular forms using a definition that, compared to \cite{hidaon}, is more suited to $p$-adic interpolation. Most of the set-up will be similar to that of $loc.$ $cit.$ but with the crucial difference that the actions are `shifted' from the infinite places to the places above $p$.	We recall here that $D/F$ is a totally definite quaternion algebra split above $p$.
	
\begin{nota}
	\begin{enumerate}[(a)]

	\item Let $G_D/\QQ_p:=\Res_{F/\QQ}(D^{\times}) \times \QQ_p$, which is a connected reductive linear algebraic group over $\QQ_p$ (via our choice of splitting). Let $T$ be the standard maximal torus, $B$ the standard Borel subgroup, and the unipotent radical $N$. We denote by $\ov{B}$ and $\ov{N}$ be the opposite Borel and opposite unipotent radical.
	Let $I \subset G_D(\ZZ_p)$ be the standard Iwahori subgroup in good position with respect to $B$ (in good position means that $B,N,T,G_D,\ov{N}$ have fixed compatible integral models over $\ZZ_p$).

	\item For $m \in \ZZ_{\geq 1}^{\Sigma_p}$, set $$I_m=\left \{ \left ( \begin{matrix} a&b\\ c&d \end{matrix} \right) \in G(\ZZ_p) \mid c \in \pi^m \OO_p \right \},$$ with $I=I_1=I_{(1,\dots,1)}$ and let $\ov{I}_m=\ov{N}(\ZZ_p) \cap I_m$.
	
	Furthermore, we set $$T^{+}=\left \{ t \in T(\QQ_p) \mid t N(\ZZ_p) t^{-1} \subseteq N(\ZZ_p) \right \}=\left \{\left ( \begin{matrix} a&0\\ 0&b \end{matrix} \right) \in T(\QQ_p) \mid ab^{-1} \in \OO_p \right \}. $$ With this we define the semigroup $\Delta=IT^+I$. 	Note that the Iwahori decomposition tells us that $I=\ov{I}_1 T(\ZZ_p) N(\ZZ_p),$ and hence any $\delta \in \Delta$ can be written uniquely as $\delta=\ov{n}_\delta t_\delta n_\delta$ with $\ov{n}_\delta \in \ov{I}_1, t_\delta \in T^+, n_\delta \in N(\ZZ_p)$.

	\end{enumerate}
\end{nota}

\begin{defn}\label{vk}
	Let $(k,r,n,v,w)$ be a weight tuple with $(k,r) \in \ZZ_{\geq 0}^{\Sigma_\infty} \times \ZZ$. Let $V_k$ be the $L$-vector space with basis of monomials $\prod_{\vv \in \Sigma_\infty}Z_\vv^{m_\vv}$, with $m \in  \ZZ_{\geq 0}^{\Sigma_\infty}$, $0 \leq m_\vv \leq k_\vv-2$. We define a right action of $\Delta=IT^{+}I$ on this space as follows: for $\g=(\g_{\ps})_{\ps \in \Sigma_p}=\left ( \begin{smallmatrix} a_\ps&b_\ps\\ c_\ps&d_\ps \end{smallmatrix} \right)_\ps \in \Delta$, let $$\g: \prod_{\vv \in \Sigma_\infty}Z_s^{m_s} \longmapsto \prod_{\vv \in \Sigma_\infty}(c_{\vv} Z_\vv+d_{\vv})^{n_\vv} \det(\g_{\vv})^{v_{\vv}} \left ( \frac{a_{\vv} Z_\vv+b_{\vv}}{c_{\vv} Z_s +d_{\vv}} \right )^{m_\vv}$$
	Note that here (following \cite{buzeig}) we have adopted the notation that for $a_\ps$ (resp.\ $b_\ps,c_\ps,d_\ps$) we let $a_\vv$ (resp.\ $b_\vv, c_\vv,d_\vv$) denote the image of $a_\ps$ under the corresponding map $\ii \circ \ii_\vv$  for $\vv \in \Sigma_\ps$. Let $V_{n,v}(L)$  denote the resulting $\Delta$-module.\footnote{Note that since we have chosen weights such that $n+2v$ is parallel, $\OO_F^{\times,+}$ will act trivially, when embedded diagonally into $I$ via $\OO_F \to \OO_p \to M_2(\OO_p).$}
\end{defn}

\begin{defn}Let $\K$ be an open compact subgroup of $G_D(\AA_f)$, such that its image under the projection $\K \to D_p^{\times}$ lies in $I_1$  and  let $(k,r,n,v,w)$ be a weight tuple with $(k,r) \in \W^G$ a classical weight. The space of {\it quaternionic modular forms} over $D$ of weight $k$ and level $\K$, denoted $S_{k,r}^{D}(\K)$, is the space of functions $$f:G_{D}(\AA_f) \lra V_{n,v}(L)$$ such that:

	\begin{enumerate}[$(a)$]
		\item For $\g \in G_{D}(\QQ)$, we have $f(\g g)=f(g)$ for all $g \in G_D(\AA)$.
		\item For $u \in \K$ we have $f(g)=f(gu^{-1}) \cdot u_{p}$ for all $g \in G_{D}(\AA)$, where $u_p$ denotes the $p$-part of $u$.
	\end{enumerate}
\end{defn}

We will only be interested in the following open compact subgroups of $G_D(\AA_f)$:

\begin{defn}\label{cong}	
 Let $\nn=\prod_v \gothq_{v}^{e_v}$ be an integral ideal. For $D$ a totally definite quaternion algebra\footnote{Note that if we take $D=M_2(F)$ one obtains the usual level structures for Hilbert modular forms.} with $(\disc(D),\nn)=1$, we fix splitting at all primes dividing $\nn$ of $D$. Then we define
\begin{align*}
\K_1(\nn)&:= \left \{\g \in (\OO_D \otimes \wh{\ZZ})^\times \mid \g \equiv \left ( \begin{smallmatrix} *&*\\ 0&1  \end{smallmatrix} \right) \mod \nn \right \},\\
\K_0(\nn)&:= \left \{\g \in (\OO_D \otimes \wh{\ZZ})^\times \mid \g \equiv \left ( \begin{smallmatrix} *&*\\ 0&*  \end{smallmatrix} \right) \mod \nn \right \},\\
\K(\nn)&:= \left \{\g \in (\OO_D \otimes \wh{\ZZ})^\times \mid \g \equiv 1 \mod \nn \right \}.
\end{align*}		

\end{defn}

\end{subsection}

\begin{subsection}{Overconvergent spaces}

We are now going to define the spaces of overconvergent modular forms for $D$. For this we need to find a larger $\Delta$-module containing $V_{n,\v}(L)$, so work with the spaces of locally analytic functions.

 	Let $X \subset \QQ_p^{s}$ be open and compact. 
 \begin{defn}\label{lan}
 	For a finite extension $L/ \QQ_p$, we say a function $f:X \to L$ is {\it{$L$-analytic}} if it can be expressed as a converging power series $f(x_1,\dots,x_s)=\sum_{t_1,\dots,t_s} \a_{t_1,\dots,t_s} (x_1-a_1)^{t_1}\cdots (x_s-a_s)^{t_s},$ for $\a_{t_1,\dots t_s} \in L$, and some $(a_1,\dots,a_s) \in X$. We say it is {\it{algebraic}} if almost all $\a$'s are zero. 
 \end{defn}
 
 \begin{defn}
 	For $m \in \ZZ^r_{\geq 0}$, let $\A(X,L,m)$ be the $L$-vector space of {\it $m$-locally analytic functions}, i.e., functions that are analytic on balls of radius $p^{-m}$ covering $X$.
 	Then $\A_m(X,L)$ is a $p$-adic Banach space when $X$ is  compact. We let $\A(X,L)= \bigcup_{m \geq 0} \A(X,L,m).$
 	This is the space of functions $f: X \to L$ that are $m$-locally $L$-analytic for some $n$.
 \end{defn}

	We now define the $\Delta$-modules that we will be interested in.

\begin{defn}
	
	We begin by identifying $\OO_p$ with an open compact subset of $\QQ_{p}^{g}$ compatible with the identification of $I$ as an open compact of $\QQ_p^{4g}$. We then consider $\A(\OO_p,L)=\bigcup_{m \geq 0} \A(\OO_p,L,m).$ This is a $\Delta$-module with the following action. For $(\k,r,n,v,w)$ a weight tuple with $(\kappa,r) \in \W^G(L)$, $f \in \A(\OO_p,L)$, $\g=\left ( \begin{smallmatrix} a&b\\ c&d \end{smallmatrix} \right) \in \Delta $ and $z \in \OO_p$, let $$(f \cdot\g)(z)= n(cz+d)v(\det(\g))f\left (\frac{az+b}{cz+d} \right).$$ We denote this module by $\A_{n,v}(\OO_p,L)$. 
	
\end{defn}

	\begin{lem}
	For\/ $(\k,r) \in \W^G$ there exists a smallest\/ $m(\kappa,r)$, such that for all\/ $m \geq m(\kappa,r)$, $(\kappa,r)$ is\/ $m$-locally analytic.
\end{lem}

\begin{proof}
	See \cite[Lemma 3.2.5]{urb}.
\end{proof}

From this it follows that  $\A_{n,v}(\OO_p,L)=\bigcup_{m \geq m(\kappa,r)}\A_{n,v}(\OO_p,L,m),$ where $\A_{n,v}(\OO_p,L,m)$  is the $\Delta$-module $\A(\OO_p,L,m)$ with the action defined as above. More generally, since we wish to consider families of modular forms, one can extend this definition as follows:

\begin{defn}
	If $\V$ is an affinoid subdomain of $\W^G$ defined over a finite extension $L/\QQ_p$ and $(\kappa^{\V},r^\V)$ is the restriction of the universal character to $\V$, then we define $\A_{\V}(\OO_p,L):=\A_{n^{\V},v^\V}(\OO_p,L),$ with the action of $\Delta$ defined analogously where $(\k^\V,r^\V,n^\V,v^\V,w^\V)$ is a weight tuple.
\end{defn}

It follows from \cite[Lemma 3.4.6]{urb}, that there exists a smallest integer $m(\V)$ such that $(\kappa^{\V},r^\V)$ is $m(\V)$-analytic. Moreover, $\A_{\V}(\OO_p,L)=\bigcup_{m \geq m(\V)} \A_{\V}(\OO_p,L,m).$

	\begin{defn}Let $(\k,r,n,v,w)$ be a weight tuple with $(\kappa,r) \in \W^G(L)$ and $\K$ be an open compact subgroup of $G_D(\AA_f)$, such that its image under the projection $\K \to D_p^{\times}$ lies in $I_{m}$ for some $m \geq 1$  and $t \in \ZZ_{\geq 0}^{\Sigma_p}$ is such that  $t+m \geq m(\kappa)$ . The space of overconvergent quaternionic modular forms of weight $\kappa$, level $\K$ and radius of overconvergence $p^{-t}$, denoted  $S_{\kappa,r}^{D,\dagger}(\K,t)$ is the space of functions $$f:G_{D}(\AA_f) \lra \A_{n,v}(\OO_p,L,t)$$ such that:
		\begin{enumerate}[$(a)$]
		\item For $d \in G_{D}(\QQ)$, we have $f(d g)=f(g)$ for all $g \in G_D(\AA)$.
		\item For $\g \in \K$ we have $f(g)=f(g\g^{-1}) \cdot \g_{p}$ for all $g \in G_{D}(\AA)$, where $\g_p$ is the $p$-part of $\g$.
	\end{enumerate}
	If $\V \subset \W^G$ is an affinoid subdomain defined over $L$ and $t \geq m(\V)$, then  define $S_{\V}^{D,\dagger}(\K,t)$ to be the space of functions  $f:G_{D}(\AA_f) \lra \A_{\V}(\OO_p,L,t)$ satisfying $(a), (b)$ above. Lastly, taking the limit over $t$ we obtain Fr\'echet spaces $S_{\kappa,r}^{D,\dagger}(\K)$.
	
\end{defn}

\end{subsection}

\begin{subsection}{Hecke operators and the Control Theorem}\label{Hecke}

Following Section 12 of \cite{buzeig}, we define the Hecke operators on these spaces.

\begin{defn}For $\K'=U_*(\nn) \cap U_i(\pi^s)$  with $\nn$ coprime to $\pi$, we call $\nn$ the {\it tame level} and $\pi^s$ the {\it wild level}.
\end{defn}	

\begin{nota}\label{hecke} If $\vv$ is a finite place of $F$, such that $D_\vv$ is split, then let $\eta_\vv \in G_D(\AA_f)$ be the element which is the identity at all places different from $\vv$ and at $\vv$ it is the matrix $ \left ( \begin{smallmatrix} \pi_\vv&0\\ 0&1 \end{smallmatrix} \right)$, for $\pi_\vv$ a uniformiser of $F_\vv$.  In order to ease notation later on, when $\vv|p$ we choose the same uniformisers as we had before. 
\end{nota}

\begin{defn}\label{heckealg} Let $U$ have tame level $\nn$ and wild level $\pi^s$. For each $\vv$ as above, we define the Hecke operators $T_\vv$ as the double coset operators given by $[\K \eta_\vv \K]$. Moreover, if $\vv$ is coprime to level, then we can regard $\pi_\vv$ as an element of the centre of $D_f^\times$ and we denote by $S_\vv$ the operator $[\K \pi_\vv \K]$. Lastly, for each $\ps \in \Si_p$ let $U_{\ps}$ denote the operator $T_{\ps}$ and let $U_p=\prod_{\gothp \in \Sigma_p} U_{\ps}^{e_{\ps}}$. We denote by $\pmb{T}=\pmb{T}^D(\K)$, the Hecke algebra generated by the operators\footnote{Note that these operators are independent of choice of uniformiser for $\v$ not dividing $p$. } $T_{\gothq}, U_{\ps}$, where $\gothq \nmid \nn\d$ with $\d=\disc(D)$ and $\ps \in \Sigma_p$.

\end{defn}

	We now want to show that the overconvergent quaternionic modular forms of small slope are classical. To do this we will follow the proof of the case $F=\QQ$ in \cite[Section 7]{kevp}. We begin with some preliminaries.

\begin{lem}The\/ $U_p$ operator acting on\/ $S_{\V}^{D,\dagger}(\K\cap \K_0(\pi^s),t)$ for\/ $s+t \geq m(\V)$ is compact.  In particular, this holds for the spaces\/ $S_{\kappa,r}^{D,\dagger}(\K\cap \K_0(\pi^s),t)$ for\/ $s + t \geq m(\kappa,r)$.

\end{lem}
\begin{proof}See \cite[Lemma 12.2]{buzeig} or \cite[Lemma 3.2.8]{urb}.
	
\end{proof}

\begin{defn}
	Let $(\k,r)$ be an algebraic weight and let  $\k_i=(k_1,\dots,k_{i-1}, 2-k_i,k_{i+1},\dots k_g ).$ Note that if $\k=2w-r$ then $\k_i=2w'-r$ where $w'_j = w_j$ for $j \neq i$ and $w'_i=v_i$. For each $i \in \{1,\dots,g\}$ corresponding to a place in $ \Sigma_\infty$, we define a map $\Theta_i: S_{\k,r}^{D,\dagger}(\K,0) \lra  S_{\k_i,r}^{D,\dagger}(\K,0)$ by setting $\Theta_i(f)(h)=\frac{\partial^{k_i-1}f(h)}{\partial z_i^{k_i-1}}$ for $h \in G_D(\AA_f)$.
\end{defn}
Note that $f(h) \in \A_{n,v}(\OO_p,L,0)$ so it can be written as a converging power series in variables $(z_1,\dots,z_g)$, so $\frac{\partial^{k_i-1}f(h)}{ \partial z_i^{k_i-1}}$ makes sense. Moreover, one needs to check that $\Theta_i$ is actually well-defined, but this follows at once from the simple check that for any $\g \in I$ we have $\Theta_{i}(f)|_\g=\Theta_{i}(f|_\g)$.

\begin{thm}[\bf Control Theorem] Let\/ $\K'=U_*(\nn)$ with $(\nn,\pi)=1$ and\/ $\K=\K' \cap \K_1(\pi^s)$ for\/ $s \geq 1$ and let\/ $(k,r)$ be a classical weight. Let\/ $f \in S_{k,r}^{D,\dagger}(\K,t)$ be an eigenform for each\/ $U_{\ps_i}$ with eigenvalue\/ $\a_{\ps_i}$. If for each\/ $\ps_i|p$ we have  $$\val_p(\a_{\ps_i}) < \frac{v_{\ps_i}(k,r)+ \min_{j \in \Si_{\ps_i}} \{k_j-1\}}{e_{\ps_i}},$$ where\/ $e_{\ps_i}$ is the ramification degree, then\/ $f \in S_{k,r}^{D}(\K)$  (in other words, $f$ is classical).
	
\end{thm}

\begin{proof}We will only sketch the proof, but the full details\footnote{Up to normalization} can be found in \cite[Theorem 2.3]{Yama}. First note that if $\Theta_i(f)=0$ for all $i$ then $f$ must in fact be classical. The task is now to give a criterion for $f$ to be in this kernel based only on the slope of $f$. Now, let $U_{\gothp_i}^0=\pi_{\gothp_i}^{-v_{\ps_i}(k,r)}U_{\ps_i}$ which has operator norm $\leq 1$. Then any eigenform of $U_{\gothp_i}^0$ with negative slope must in fact be zero. Now $\Theta_i$ sends $U_{\gothp_i}^0$-eigenforms of slope $h$ to  $U_{\gothp_i}^0$-eigenforms of slope $h - \frac{ \min_{j \in \Si_{\ps_i}} \{k_j-1\}}{e_{\ps_i}}$, from which one can deduce the result. 
	
\end{proof}

	Using the above and the Eigenmachine we can construct the eigenvariety associated to overconvergent quaternionic modular forms for $D/F$. 

\begin{thm}\label{buzeig} Let\/ $\K=\K' \cap \K_0(\pi)$ with $\K'$ having level $\nn$  coprime to\/ $\pi$. Let $\CZ$ be the spectral variety defined as usual and\/ $\pmb{T}=\pmb{T}^D(\K)$ as defined in \ref{heckealg}. Lastly, let\/ $S^{D,\dagger}(\K)$ be the coherent sheaf given the nuclear Frechet spaces \/ $S_{{\V}}^{D,\dagger}(\K)$ where\/ $\V$ is an affinoid with a morphism\/ $\V \to \W^G$. Associated to the eigendata of\/ $(\W^G,S^{D,\dagger}(\K),\pmb{T},U_p)$ we have an eigenvariety\/ $\CX_D(U)$ which is equidimensional of dimension\/ $g+1$ and satisfies the conditions of Theorem \ref{th1}.

\end{thm}

\begin{proof}The existence of such an eigenvariety and the fact that it is equidimensional  follows from \cite[Section 13]{buzeig}. The fact that is it is equidimensional of dimension $g+1$ is due to the weight space that we have used.
\end{proof}

\end{subsection}

\end{section}

\begin{section}{\textbf{ p-adic Langlands functoriality}}
\setcounter{subsection}{1}
In this section we will relate the different eigenvarieties we have defined in the previous section. In particular, we prove the following:

	\begin{thm}\label{thm71}Let $D$ be a totally definite quaternion algebra with $\disc{D}=\d$ and $\nn$ an ideal with $(\nn,\d)=1$.  Let $\CX_{G}:=\CX_{G}(U_1(\nn \d))$ and $\CX_{D}:=\CX_D(U_1(\nn))$ be as in Theorems~\ref{AIPeig} and~\ref{buzeig} respectively, with $U_1(\nn\d)$ a level whose associated moduli problem is representable\footnote{Meaning that the moduli problem of HBAV with a $\mu_{\nn \d}$-level structure is representable.} and let $p$ be unramified. 
	Then these eigenvarieties are reduced and the classical Jacquet-Langlands correspondence can be interpolated to obtain a closed immersion $\CX_D \hookrightarrow  \CX_G$ satisfying the properties of Theorem \ref{int}. 
	
\end{thm}

\begin{cor}\label{cor62}
	If $g=[F:\QQ]$ is even, then taking $D$  (totally definite) with $\d=1$, the closed immersion given by Theorem \ref{thm71} becomes an isomorphism.
\end{cor}

We will derive Theorem~\ref{thm71} from Theorem~\ref{int} (the Interpolation theorem). To this end, we need to exhibit a very Zariski dense set $X \subset \W^G$ on which we can put classical structures for both sets of eigenvariety data. The set of all classical weights (see Definition \ref{cw}) is such a candidate. The fact that it is a very Zariski dense subset of $\W^G$ is a well-known fact but we include its proof for the sake of completeness. This requires the following lemma.

\begin{lem}\label{ir} If\/ $W$ is a non-empty rigid space. Then $W$ is irreducible if and only if the only analytic subset $Z \subset W$ which set-theoretically contains a non-empty admissible open of\/ $W$ is $Z=W$.

\end{lem}

\begin{proof}See \cite[Lemma 2.2.3]{conr99}.

\end{proof}

\begin{prop}Let $X$ be the set of classical weights, then $X$ is very Zariski dense in $\W^G$.

\end{prop}

\begin{proof}This is a simple generalization of \cite[Proposition 6.2.7]{chen2} or \cite[Lemma 4.1]{Ram}. By  \ref{wtspstr} we have $$\W^G \cong H^{\vee} \times B(1,1)^{g+1} \cong \bigsqcup_\chi \W_\chi,$$ where we index over the elements of  $H^\vee $.
	Let $\k_{\psi}$ be a classical weight with  $\kappa=2w-r$ and $\psi$ as  Definition \ref{cw}. Then under the above isomorphism $$\kappa_{\psi} \mapsto \left (\wh{\kappa}_{\psi}, \left [\prod_j^g (1+p)^{w_j} \right ], (1+p)^{r} \right ),$$ where $\wh{\kappa}_{\psi }$ denotes the restriction to  $H^\vee $ (note that $\k_{\psi} \in \W^{G}(E)$, with $E=\QQ_p[\psi]$). Assume that $\kappa_{\psi} \in \W_\chi$ for some $\chi$ and take any irreducible admissible affinoid open $V \subset \W^G$ that contains $\kappa_{\psi}$. Then  $V \subset \W_\chi$ and moreover, since $V(E)$ is open, there exists $\pmb{s}=(\un{s},s') \in \QQ_{>0}^{\Sigma_\infty} \times \QQ_{>0}$ such that  the closed ball of radius $\pmb{s}$ around $\k_{\psi}$ is contained in $V$, i.e., $$B[\kappa_{\psi},\pmb{s}]:=\prod_{j}^g B[w_j,s_j] \times B[r,s'] \subset V. $$
	By Lemma \ref{ir}, we see that if $B[\kappa_{\psi },\pmb{s}] \cap X$ is Zariski dense in $B[\kappa_{\psi },\pmb{s}]$, then $V(E) \cap X$ is Zariski dense in $V$, which is what we want to prove. So we are reduced to showing that  $B[\k_{\psi },\pmb{s}](E) \cap X$ is Zariski dense in $B[\k_{\psi},\pmb{s}]$. To see this, let $\kappa,\kappa' \in X$, then
	
	\begin{align*}
	|\kappa -\kappa'|&= \max_{i} \{ |(1+p)^{w_i}-(1+p)^{w'_i}|_p,  |(1+p)^{r}-(1+p)^{r'}|_p \} \\
	&= \max_i  \{ |(1+p)^{w_i - w'_i}-1|_p,  |(1+p)^{r-r'}-1|_p \}.
	\end{align*}
	So taking $\kappa'_{\psi} \in X$, with $\k'=2w'-r'$ is such that:
	\begin{itemize}
		\item for $N$ large enough, $w_i \equiv w'_i \mod (p-1)p^N$, $r \equiv r' \mod (p-1)p^N$;
		\item $\k_{\psi}$ and $\k'_{\psi}$ lie in the same component of the weight space.
	\end{itemize}
	Then we can easily see that $B[\k_{\psi},\pmb{s}]$ contains infinitely many elements of $X$, hence we get the result.
\end{proof}

	\begin{defn}
	Let $\CZ \subset \W^G \times \AA^1$ be the spectral variety defined by $\Fred_\CM(U_p)$ for $\CM$ (as usual) the coherent sheaf on $\W^G$ of overconvergent Hilbert modular forms on $D$ or $M_2(F)$ with a classical structure $\CM^{cl}$. We call a point $z \in \CZ$ classical if its projection to $\W^G$ is a classical weight and if  $\det(1-TU_p|_{\CM^{cl}})$ vanishes at $z$. We denote these points by $\CZ^{cl}$.
\end{defn}

\begin{rmrk}
	Note that if $\CM$ is given by the `spreading out' of the spaces of overconvergent Hilbert modular forms. Moreover,  $z=((\kappa_z,r_z),\a) \in \W^G \times \AA^1$  is classical if $(\k_z,r_z)$ is a classical weight and there exists a classical Hilbert modular form with $U_p$ eigenvalue $\a^{-1}$.
\end{rmrk}

\begin{prop}\label{handens}
	The subset $\CZ^{cl}$ of\/ $\CZ$ is a very Zariski dense subset.
\end{prop}

\begin{proof}
	This follows from the proof of \cite[Proposition 3.5]{chenjlc}. The basic idea is to use the fact that $X$ and $X_h$ are very Zariski dense, together with the fact that the admissible cover of $\CZ$ as given by \cite[Section 4]{buzeig} is finite flat over its projection to weight space.
\end{proof}

Let $\nn$ be an ideal of $\OO_F$ with $(\nn,\d)=1$ and $\pi \nmid \nn\d$, where $\d=\disc(D)$. Let  $\K^D=U_1(\nn \pi)$ and set $\pmb{T}^D(\K^D)$ to be the Hecke algebra.\footnote{Note that the Hecke algebra consists of all Hecke operators away from $\d$.} Let $U^G$ be the corresponding level structure when one takes $D=M_2(F)$, which gives the level structure in the Hilbert modular form case. By fixing a splitting at places away from $\d$, we let $\pmb{T}^D$ act on the spaces of Hilbert modular forms. Therefore, throughout this section we denote $\pmb{T}^D$ simply by $\pmb{T}$.

\begin{thm}[\textit{The Jacquet--Langlands correspondence}]\label{thm:jl}Let $(k,r)$ be a classical weight in $\W^G$ and $U^D,U^G$ as above. There is an isomorphism \[S_{k,r}^{D}(\K^D) \overset{\sim}\to S_{k,r}^{\d \text{-new}}(U^G \cap U_1^G(\d))).\]
\end{thm}

	\begin{rmrk}\label{rem:jl}		
	We note that, for $g$ even, we can pick the quaternion algebra $D$ to be totally definite with $\d = 1$. Now, by fixing a splitting we can identify $U^D$ and $U^G$, which we will simply denote by $U$. In this case the classical Jacquet-Langlands correspondence gives an isomorphism of Hecke modules $S_{k,r}^{D}(\K) \overset{\sim}\to S_{k,r}(\K).$ However, for $g$ odd, since $D$ is totally definite, we must have $\d \neq 1$. In this case, we have an isomorphism of Hecke modules $$S_{k,r}^{D}(\K^D) \overset{\sim}\lra S_{k,r}^{\d \text{-new}}(U^G(\d)) \hookrightarrow S_{k,r}(U^G(\d)), $$ where $U^G(\d)=U^G \cap U_1^G(\d)$.
\end{rmrk}

	Theorem \ref{thm71} then follows from Theorem \ref{thm:jl-interpolate} (below) together with Lemma \ref{u}:

\begin{thm}\label{thm:jl-interpolate} Let $\CX_G$ and $\CX_D$ be the eigenvarieties associated to the eigendata $$\mathfrak{D}_1=(\W^G,S^{G,\dag}(\K^G(\d)),\pmb{T},U_p) \text{ and  }\mathfrak{D}_2=(\W^G,S^{D,\dagger}(U^D),\pmb{T},U_p)$$ respectively and let $p$ be unramified. Then we can interpolate the classical Jacquet-Langlands correspondence and obtain a closed immersion $\iota_D:\,\CX_D^{red} \hookrightarrow \CX_G^{red}$. 
	
\end{thm}

	\begin{proof}We will prove this using Theorem \ref{int}. Let $X$ be the set of classical weights, whose elements we will denote by $k$.  We now define classical structure on $X$. For each $(k,r) \in X$, let $\CM_{G,k,r}^{cl}$ and $\CM_{D,k,r}^{cl}$ be the $\pmb{T}$-modules $S_{k,r}(U^G(\d))$ and $S_{k,r}^{D}(\K^D)$ respectively of classical cusp forms of weight $k$ and level $U^G(\d),U^D$ respectively.  We need to check that this is indeed a classical structure.
	Pick $h \in \RR_{\geq 0}$. Then the set of $(k,r) \in X$ such that $ S_{k,r}^{G,\dagger}(U^G(\d))^{\leq h} \subset \CM_{G,k,r}^{cl}$ contains all $(k,r) \in \W^G$, such that $h < v_p(k,r) + \min_{i \in \Sigma_\infty} \{k_i-1 \}$ by the Control Theorem and hence satisfies the properties of Definition \ref{51}. Recall that the superscript $\leq h$ denotes  slope decomposition with respect to $U_p$.
	
	Similarly, if $(k,r) \in X$ is such that\footnote{Note that we are in the case where $p$ is unramified.}  $h < v_{p}(k,r)+ \min_{i \in \Sigma_\infty} \{k_i-1 \},$ then $ S_{k,r}^{D,\dagger}(\K^D)^{\leq h} \subset \CM_{D,k,r}^{cl}.$ It follows that we again have a classical structure. Now, as a consequence of  the classical Jacquet-Langlands correspondence we have that $\det \left (1-U_pX_{\mid \CM_{D,k,r}^{cl}} \right )$ divides $\det \left ( 1-U_pX_{\mid \CM_{G,k,r}^{cl}}\right ).$ Hence we can apply Theorem \ref{int} to obtain the closed immersion  $\iota_D:\,\CX_D^{red} \hookrightarrow \CX_G^{red}$. 
	
\end{proof}

Now observe that if $g$ is even, then we can pick $D$ to be totally definite {\it and} have $\d=1$ (i.e. trivial discriminant). Then the classical Jacquet-Langlands correspondence gives that $\CM_{G,k,r}^{cl} \cong \CM_{D,k,r}^{cl}$ at classical weights, and thus $$\det(1-U_pX\mid \CM_{D,k,r}^{cl})=\det(1-U_pX\mid \CM_{G,k,r}^{cl})$$ therefore Corollary \ref{coro} gives us an isomorphism $\CX_D^{red} \cong \CX_G^{red}$. This proves most of Corollary \ref{cor62}, it only remains to show that the eigenvarieties are reduced which we will follow from Lemma \ref{red}.

\begin{prop}\label{sem} Fix $h \in\RR_\geq 0$. There is a Zariski dense subset $X' \subset X$  (of\/ $\W^G$) such that for all $k \in X'$, the\/ $\pmb{T}$-module $\CM^{cl,\leq h}_{G,k,r}$ is semisimple. 
\end{prop}

This result will be consequence of the two lemmas below. We begin by noting that the classical Jacquet-Langlands correspondence gives us that if $\CM_{G,k,r}^{cl,\leq h}$ is a semisimple $\pmb{T}$-module, then so is $\CM_{D,k,r}^{cl,\leq h}$. To ease notation, we let \[R_{k,r}^h:=\CM_{G,k,r}^{cl,\leq h}= S_{k,r}(U^G(\d))^{\leq h}.\] Now, since we are working with classical Hilbert modular forms, the action of the Hecke operators can be described by their action on $q$-expansions. Next we note that the only Hecke operators that might not be semisimple are the $U_{\gothp_i}$, for $p\OO_F=\prod_i \gothp_i$. This is because all the other operators are normal (commute with their adjoints), so they are semisimple. Hence we must  show that for each $i$, the operators $U_{\gothp_i}$ act semisimply on the space of cusp forms of slope-$\leq h$. In fact we shall show that $U_{\gothp_i}$ acts semisimply on $R_{k',r'}^{ h}$ for a Zariski dense subset of  $ X' \subset X$.  Lastly, we need to relate slope decomposition of $R_{k,r}^{h}$ with respect to $U_p$, to the slope decompositions with respect to the $U_{\gothp_i}$. To do this we have the following:

	\begin{lem}\label{7.8}
	Let $S$ be a Banach space on which we have pairwise commuting operators\/ $U_i$ for $i=1,\dots,n$, all of which have operator norm $\leq 1$ (which means they have positive slopes) and such that\/ $U=\prod_i U_i$ is a compact operator. Then the slopes of the $U_i$ operators acting on the space\/ $S^{\leq h}$ (this is the slope decomposition with respect to $U$) are all $\leq h$.
\end{lem}

\begin{proof}By definition we have that $S^{\leq h}$ is a finite dimensional subspace of $S$. Therefore by choosing a basis we can view the $U_i$ operators as matrices. Now since the $U_i$ are pairwise commuting operators, we can simultaneously upper triangularize them (after possibly extending the base field). From this it follows that the eigenvalues of $U$ acting on $S^{\leq h}$ are the product of the eigenvalues of the $U_i$. 
	
	Now since the slopes of an operator are simply the $p$-adic valuation of its eigenvalues, we have that on $S^{\leq h}$ the slopes of $U$ are the sum of the slopes of the $U_i$ operators and therefore, since they all have positive slopes, it follows that the slopes of the $U_i$ acting on $S^{\leq h}$ are all $\leq h$ as required.
	
\end{proof}
	After renormalizing our operators, we can apply this to our situation to see that since $U_p=\prod_i U_{\gothp_i}$ is compact, then we have a slope decomposition for any $h$. Moreover, for each $h$ we have that the slope of each $U_{\gothp_i}$ acting on $R_{k,r}^{ h}$ is less than or equal to $h$. With this we can prove the following Lemma:

\begin{lem}\label{u} There is a  Zariski dense subset $ X' \subset \W^G$, such that for each $i$,  $U_{\gothp_i}$ acts semisimply on $R_{k,r}^{ h}$ for ${k,r} \in X'$.

\end{lem}

\begin{proof}
	
	Using Lemma \ref{7.8} the result is a simple generalization of the classical situation, as is done in  \cite[Theorem 3.30]{joelp}, or from the proof of  \cite[Theorem 4.2]{colem}. But for completeness we prove it here.
	
	First note that we can decompose $R_{k,r}^h$ into its $\gothp_i$-new and $\gothp_i$-old parts. The action of  $U_{\gothp_i}$ on $R_{k,r}^{h,\gothp_i{\text{-}\new}}$ is normal and hence diagonalizable. With this we are reduced to showing that this operator acts semisimply on $R_{k,r}^{h,\gothp_i{\text{-}\old}}.$ In order to prove this, it is enough to show that on each  generalized $\pmb{T}$-eigenspace of  $R_{k,r}^{h,\gothp_i\text{-old}}$ it acts semisimply. Each of these spaces will correspond to a newform $f$ of (lower) level not divisible by $\gothp_i$. Now, let  $a_{\gothp_i}=\aa(\gothp_i,f|T_{\gothp_i})$ be the $T_{\gothp_i}$ eigenvalue of $f$ and $\Psi$ its nebentypus. Since we are assuming that for each $i$, we have $(\nn\d,\gothp_i)=1$, then it follows from Atkin-Lehner Theory that each of these $\gothp_i$-old subspaces is 2-dimensional, and generated by $f$ and $f|B_{\gothp_i}$ (see \cite{lw93}, for the definition of $B_{\gothp_i}$). A simple calculation then shows that on this subspace the $U_{\gothp_i}$ operator has minimal polynomial   given by $$X^2-a_{\gothp_i}X+N_{F / \QQ}(\gothp_i)^{r+1}\Psi(\gothp_i).$$
	
	Therefore, since  $N_{F / \QQ}(\gothp_i)=p^{l_{\gothp_i}}$ (here $l_{\gothp_i}$ is the residue degree), we see that if we pick $(k,r)$, such that $r > (2h-l_{\gothp_i})/l_{\gothp_i}$, then  $h< l_{\gothp_i}(r+1)/2$. Therefore, the polynomial must have a unique root $\alpha$ with valuation $\leq h$, from which it follows that on the generalized $\pmb{T}$-eigenspace of $R_{k,r}^{h,\gothp_i\text{-old}}$ corresponding to $f$, we have that $U_{\gothp_i}$ acts as the scalar $\alpha$. Hence it is diagonalizable. This then shows that on $R_{k,r}^h$ the $U_{\gothp_i}$ operators act semisimply for $(k,r)$ large enough as required.
	
\end{proof}

	From this it follows that for any $h \geq 0$, the operators $U_{\gothp_i}$ act semisimply on $S_{k,r}(U')^{\leq h}$ for $(k,r)$ in a Zariski dense subset of $\W^G$, proving Proposition \ref{sem}. Then by Lemma \ref{red}, we have at once that $\mathscr{X}_D^{red} \cong \mathscr{X}_D$ and $\mathscr{X}_G^{red} \cong \mathscr{X}_G$, which proves Corollary \ref{cor62}.

\begin{rmrk} In light of Theorem~\ref{thm:jl-interpolate} and Remark~\ref{rem:jl}, we see that for $g$ even and $D$ totally definite with $\d =1$, we have an isomorphism of
	eigenvarieties $\iota_D:\,\CX_{D}^{red} \stackrel{\sim}{\to}\CX_G^{red}$. However, for $g$ odd and $D$ totally definite, the closed immersion $\iota_D:\,\CX_{D}^{red} \hookrightarrow\CX_G^{red}$ is never an isomorphism since $\d \neq 1$. At best, we can say that its image is the $\d$-new part of of $\CX_G^{red}$ as in the case of modular forms over $\QQ$ (cf. \cite{chenjlc}).
\end{rmrk}

	\begin{rmrk}\label{occoh}
	We note that, for $g$ odd, there are alternative constructions for the eigenvariety $\CX_{D}^{red}$. Let $D$ be the quaternion algebra ramified at all infinite places but one, with
	$\d = 1$. Then Brasca~\cite{brass} constructs an eigenvariety associated to $\CX_{D}^{red}$ from which one can use the above to obtain a closed immersion $\iota_D:\CX_{D}^{red} \hookrightarrow \CX_G^{red}$, which is an isomorphism on the closed subvariety given by parallel weights. His construction
	combines the theory of Shimura curves with work of Andreatta-Iovita-Pilloni to construct the relevant eigenvarieties. 
\end{rmrk}

\begin{rmrk}
	Using the work of \cite{Hansen}, one can construct eigenvarieties coming from overconvergent cohomology groups associated to any quaternion algebra. In this case, the resulting eigenvarieties are larger than the ones we have constructed, in particular, they may not be equidimensional of dimension $g+1$. Using a more general version of the Interpolation theorem as in $loc. cit.$ one can obtain closed immersions between the `cores' (as defined in $loc.cit.$) of the relevant eigenvarieties.
\end{rmrk}

\end{section}

\bibliographystyle{abbrv}

\bibliography{biblio(3)}

\end{document}